\documentclass[12pt]{article}

\usepackage{amssymb,amsmath}
\usepackage{amsbsy}
\usepackage{amsthm}
\usepackage{mathrsfs}
\usepackage{verbatim}
\usepackage[colorlinks]{hyperref}
\usepackage{mathdots}
\usepackage{graphicx,subfigure}
\usepackage[english]{babel}
\usepackage[utf8]{inputenc}
\usepackage{tikz}

\usetikzlibrary{backgrounds,fit,decorations.pathreplacing}
\usetikzlibrary{positioning}
\usetikzlibrary{calc,through,chains}
\usetikzlibrary{arrows,shapes,snakes,automata, petri}

\usepackage{authblk}

\newtheorem{Theorem}[equation]{Theorem}
\newtheorem{Corollary}[equation]{Corollary}
\newtheorem{Lemma}[equation]{Lemma}
\newtheorem{Proposition}[equation]{Proposition}

\theoremstyle{definition}
\newtheorem{Definition}[equation]{Definition}
\newtheorem{Example}[equation]{Example}

\theoremstyle{remark}

\newtheorem{Remark}[equation]{Remark}

\numberwithin{equation}{section}
\numberwithin{figure}{section}

\newcommand{\A}{{\mathbb A}}
\newcommand{\PP}{{\mathbb P}}
\newcommand{\C}{{\mathbb C}}
\newcommand{\Z}{{\mathbb Z}}
\newcommand{\Q}{{\mathbb Q}}
\newcommand{\R}{{\mathbb R}}

\newcommand{\N}{{\mathbb N}}

\newcommand{\mc}[1]{\mathcal{#1}} 
\newcommand{\ms}[1]{\mathscr{#1}} 
\newcommand{\mf}[1]{\mathfrak{#1}} 

\newcommand{\mt}[1]{\text{#1}}

\DeclareMathOperator{\Sym}{Sym}
\DeclareMathOperator{\Skew}{Skew}
\DeclareMathOperator{\Mat}{Mat}

\newcommand{\ud}{\,\mathrm{d}}
\newcommand{\ue}{\ud \varepsilon}

\begin{document}

\title{Regular Unipotent Invariants on the Wonderful Compactification of Complete Binary Forms}

\author[1]{Mahir Bilen Can}
\author[2]{Roger Howe}
\author[3]{Michael Joyce}
\affil[1]{Tulane and Yale Universities}
\affil[2]{Yale University}
\affil[3]{Tulane University}

	\date{\today}
	\maketitle

\begin{abstract}
Using our preliminary calculations from \cite{CHJ13}, we investigate the topology of the
closure in a wonderful compactification of the set of unipotent-invariant bilinear forms.
\end{abstract}

{\em Keywords: Wonderful compactifications, Springer fibers.}

{\em MSC-2010: 15B10, 05E18, 22E46}

\section{\textbf{Introduction}}

Let $G$ denote the special linear group $\mt{SL}_n$ over $\C$, the field of complex numbers.
Let $\mc{U} \subset G$ denote the set of all unipotent elements, which forms an
irreducible, closed subvariety of $G$. Consider the incidence
variety $X= \{ (gB,x)\in G/B\times G :\ g^{-1}xg \in \mc{U} \}$.
It is well known that $\mc{U}$ is singular; however, the second projection $pr_2 : X \rightarrow G$
is a resolution of singularities for $\mc{U}$ (see \cite{Springer68} and \cite{Steinberg76}). The fibers of $pr_2$, known as
{\em Springer fibers}, have many remarkable combinatorial and geometric properties.
For example, if $u\in G$ is a unipotent element, then for each $i\geq 0$
the cohomology space $H^{i}(pr_2^{-1}(u), \Q )$ is a representation of the
symmetric group $S_n$, and the top non-vanishing cohomology space is an irreducible $S_n$-module
(\cite{Springer76}, \cite{Springer78}). Furthermore, all irreducible components of $pr_2^{-1}(u)$ have the same dimension
$e(u)$, which is equal to the half of the dimension of the top non-vanishing cohomology space
$H^{2e(u)} (pr^{-1}(u),\Q)$, \cite{Spaltenstein77}. These facts lead to striking combinatorial results as
in \cite{HottaShimomura,Lusztig81, Steinberg88}.

Motivated by the combinatorial significance of the fibers of $pr_2$, in our earlier paper \cite{CHJ13} we initiated a study of the sets
\begin{align}\label{earlier}
(G/K)_u = \{ gK:\ g^{-1}ug \in K \},
\end{align}
when $K=\mt{SO}_n$, the special orthogonal group, and when $K= \mt{Sp}_n$, the symplectic group.
Here $u\in G$ is a prescribed unipotent element.
Unlike the fibers $pr_2^{-1}(u)$, our sets $(G/K)_u$ are not complete.
Therefore, in this paper we continue our analysis by studying their completions in a compactification of $G/K$.

Suppose there exists an involution $\theta: G\rightarrow G$ such that $K= \{ g\in G:\ \theta(g)= g\}$ is
the fixed subgroup of $\theta$. Let $H=N_G(K)$ denote the normalizer of $K$ in $G$; in the two examples above, $H = Z(\mt{SL}_n) K$, 
where $Z(\mt{SL}_n)$ denotes the center of $\mt{SL}_n$.
The {\em wonderful embedding of $G/H$} is the unique smooth projective $G$-variety $\mc{M}=\mc{M}_{G/H}$
that contains a copy of $G/H$ as a $G$-stable open subset, and its boundary
$\mc{M} - G/H$ is a finite union of $G$-stable divisors with normal crossings \cite{DP83}.
The case where $G = G_1 \times G_1$, with $G_1$ reductive and $\theta(x,y) = (y,x)$, so that $K = \text{diag}(G_1)$ 
and $G/K \cong G_1$, has been studied previously.
In fact, the wonderful embedding in this case is used in the study of character sheaves by Lusztig  \cite{Lusztig04} and by Lusztig-He \cite{HeLusztig06}. 
The Zariski closures in the wonderful compactifications of some closely related subvarieties of $G_1$
(called ``Steinberg fibers'') are studied by He \cite{He06} and by He-Thomsen \cite{HeThomsen06}.
For a survey of these results, we recommend Springer's I.C.M. report \cite{Springer06}.

Our main focus is on the fixed locus of a regular unipotent element $u$.  (Here, regular means that the Jordan type of $u$ consists of a single block.)
Let $X$ denote the wonderful compactification $\mathcal{M}_{G/H}$ as described above corresponding to $K = \mt{SO}_n$ or $K = \mt{Sp}_n$. 
We construct a cell decomposition of the variety $X^u$ consisting of complete quadrics (respectively, complete skew forms) that are invariant under $u$.
As an application, we compute the Poincar\'e polynomial (i.e. the generating function for the dimensions of the cohomology with $\mathbb{Q}$-coefficients) 
of the variety $X^u$ in each of these cases.
The Poincar\'e polynomial corresponding to $\mt{SO}_n$ is
$\sum_{i=0}^n {n - 1 - i \choose i} x^i$, and the Poincar\'e polynomial for $\mt{Sp}_{n}$ is $(1+x)^{n/2-1}$.
(In fact, the first Poincar\'e polynomial is a polynomial analog of the $(n-1)$-st Fibonacci number.)

Our paper is organized as follows.  After giving some notational conventions in the next section, in Section \ref{SS:unipotentinvariantmatrices} 
we review some results from \cite{CHJ13} that explicitly describe which matrices are invariant under the action $u \cdot A = (u^{-1})^{\top} A u^{-1}$ of a unipotent matrix $u$.
In Section \ref{SS:sl2triplet} we introduce one parameter subgroups that we use in the cell decompositions of Section \ref{S:celldecompositions}, 
and in Section \ref{SS:embeddings} we give a brief overview of the construction of the wonderful compactifcation, focusing on the special cases of
complete quadrics (Section \ref{S:complete quadrics}) and complete skew forms (Section \ref{S:complete skew forms}) in detail.
In Section \ref{S:fixedsubvarieties}, we describe a certain decomposition of unipotent fixed elements of the varieties of complete quadrics and of complete 
skew forms that we prove in Section \ref{S:celldecompositions} give a Bia{\l}ynicki-Birula cell decomposition.  
Finally, we prove the assertions made above about the Poincar\'e polynomials.

\vspace{1cm}
\textbf{Acknowledgements} The first author is partially supported by the Louisiana Board of Regents
enhancement grant.

\section{\textbf{Notation}}
\label{S:Notation}

We use $\N$ to denote non-negative integers, and if $n\in \N$ is non-zero,
then we denote by $[n]$ the set $\{1,2,\dots, n\}$.
All varieties and spaces in this manuscript are defined over $\C$ and are not necessarily irreducible.
Our basic list of notation is as follows:
\begin{eqnarray*}
\Mat_n: &&\ \text{space of}\ n\times n \ \text{matrices}\\
\mt{SL}_n \subset \Mat_n: &&\ \text{invertible matrices with determinant 1}\\
\Sym_n: &&\ \text{space of}\ n\times n \ \text{symmetric matrices}\\
\Sym_n^0 \subset \Sym_n: &&\  \text{invertible symmetric matrices }\\
\Skew_n: &&\ \text{space of}\ n\times n \ \text{skew-symmetric matrices}\\
\Skew_n^0 \subset \Skew_n: &&\  \text{invertible skew-symmetric matrices }
\end{eqnarray*}
Note that $\Mat_n \cong \Sym_n \bigoplus \Skew_n$.

There is a common left action of $\mt{SL}_n$ on the spaces $\Mat_n,\Sym_n$ and $\Skew_n$ given by
\begin{align}\label{Main Action}
g \cdot A = (g^{-1})^\mathsf{T} A g^{-1}.
\end{align}
It follows from Cholesky's decompositions that on $\Sym_n^0$ and $\Skew_n^0$ the action (\ref{Main Action}) of $\mt{SL}_n$
is transitive (see Appendix B of \cite{GoodmanWallach}).
It is easy to check that the stabilizer of the identity matrix $I_n \in \mt{Sym}_n^0$ is the special orthogonal group $\mt{SO}_n:=\{ A \in \mt{SL}_n:\ A A^\top = I_n \}$, 
and the stabilizer of
$$
J = \begin{pmatrix} 0 & I_n \\ -I_n & 0 \end{pmatrix} \in \mt{Skew}_{2n}^0
$$
is the symplectic subgroup $\mt{Sp}_{2n} \subset \mt{SL}_{2n}$.
Therefore, the quotients $\mt{SL}_{n}/\mt{SO}_{2n}$ and $\mt{SL}_{2n}/\mt{Sp}_{2n}$ are canonically identified with
the quasi-affine varieties $\mt{Sym}_n^0$ and $\mt{Skew}_{2n}^0$, respectively.

A {\em composition} of $n$ is an ordered sequence positive integers that sum to $n$.
For example, $(1,3,2,4,2)$ is a composition of 12. We denote by $Comp(n)$ the set of all
compositions of $n$.

A {\em partition} of $n$ is an unordered sequence of integers that sum to $n$.
Since it is unordered, without loss of generality we assume that the entries of a partition
form a non-increasing sequence.
Let $\lambda = (\lambda_1, \lambda_2, \dots, \lambda_k)$ be a partition of $n$ with $\lambda_k \geq 1$.
In this case, we call $k$ the length of $\lambda$, and denote it by $\ell(\lambda)$. Entries $\lambda_i$, $i=1,\dots, k$
are called the {\em parts} of $\lambda$.
Alternatively, we write $\lambda = (1^{\alpha_1}, 2^{\alpha_2}, \dots, l^{\alpha_l})$ to indicate that
$\lambda$ consists of $\alpha_1$ 1's, $\alpha_2$ 2's, and so on.
Zero parts or terms with zero exponent may be added and removed without altering $\lambda$.
For example, each of $(3,3,2)$, $(3,3,2,0)$, $(1^0, 2^1, 3^2)$, and $(2^1, 3^2)$ represent the same
partition $3 + 3 + 2$ of $8$.  As a special case used frequently in this paper, the notation $(k)$ indicates the partition with single part equal to $k$.

Finally, for an arbitrary matrix $B=(b_{ij})_{i,j=1}^n$, we say that $b_{ij}$ lies on the $k$-th anti-diagonal if $i+j = n+k$.
The first anti-diagonal of $B$ is the {\em main anti-diagonal}.

\section{\textbf{Preliminaries}}
\label{S:preliminaries}

\subsection{Unipotent invariant matrices}
\label{SS:unipotentinvariantmatrices}

We recall some basic structural results on unipotent fixed matrices from \cite{CHJ13}.

Let $N$ be a nilpotent matrix with entries in $\C$, $u = \exp N$.
There is an elementary but efficient criterion for determining when a matrix is fixed by $u$:
\begin{align}
u\cdot A = A \ \iff \ AN+ N^\mathsf{T} A = 0.
\end{align}

It is easy to see that the Jordan type of $u = \exp(N)$ is the same as the Jordan type of $N$.
Let $N$ and $N'$ be two nilpotent matrices that are conjugate in $\mt{SL}_n$, say $N' = S N S^{-1}$. If we set
$u = \exp(N)$ and $u' = \exp(N')$, then the fixed point loci of $u$ and $u'$ are isomorphic via
$A \mapsto S A S^\mathsf{T}$.
It follows that the fixed point space $X^u$ depends only on the Jordan type of $u$, or equivalently
only on the Jordan type of $N$.

It is well known that the Jordan classes of $n$-by-$n$ nilpotent matrices are in
bijection with partitions of $n$.  Indeed, let $N_{(p)}$ be the $p$-by-$p$ matrix
\begin{align}
N_{(p)} = \left( \begin{matrix}
0 & 1 & 0 & \cdots & 0 & 0\\
0 & 0 & 1 & \cdots & 0 & 0\\
\vdots & \vdots & \vdots & \ddots & \vdots & \vdots\\
0 & 0 & 0 & \cdots & 0 & 1\\
0 & 0 & 0 & \cdots & 0 & 0
\end{matrix} \right).
\end{align}
Then the above correspondence associates to a partition
$\lambda = (\lambda_1, \lambda_2, \dots, \lambda_k$),
to the Jordan matrix $N_{\lambda}$ given in block form by
$$
N_{\lambda} = \left( \begin{matrix}
N_{\lambda_1} & 0 & \cdots & 0\\
0 & N_{\lambda_2} & \cdots & 0\\
\vdots & \vdots & \ddots & \vdots\\
0 & 0 & \cdots & N_{\lambda_k}
\end{matrix} \right).
$$
From now on if the Jordan type of $u$ is $\lambda$, we write $X^\lambda$ in place of $X^u$.

It is straightforward to verify that a generic element $A = (a_{i,j})_{i,j=1}^n$ of $\Mat_n^{(n)}$ has the following form:
\begin{align}\label{A: n=n}
\begin{pmatrix}
0 & 0 &0 &\cdots &0 & 0 & a_{1,n} \\
0 & 0 &0 & \cdots  & 0 & -a_{1,n} & a_{2,n} \\
0 & 0 & 0 & \cdots &  a_{1,n} & -a_{2,n} & a_{3,n} \\
\vdots &  \vdots & \vdots &  \iddots & \vdots & \vdots & \vdots \\
0 & 0 & (-1)^{n-3} a_{1,n} &\cdots & a_{n-4,n} & -a_{n-3,n} & a_{n-2,n}\\
0 & (-1)^{n-2}a_{1,n} &(-1)^{n-3} a_{2,n}&  \cdots & a_{n-3,n} & -a_{n-2,n} & a_{n-1,n}\\
(-1)^{n-1} a_{1,n} & (-1)^{n-2} a_{2,n} & (-1)^{n-3} a_{3,n} & \cdots & a_{n-2,n} & -a_{n-1,n} & a_{n,n}
\end{pmatrix}
\end{align}

An analysis of the matrices of the form (\ref{A: n=n}) yields the following propositions.


\begin{Proposition}\label{P:symmetric regular case}
Let $A=(a_{i,j})_{i,j=1}^n$ be a symmetric matrix satisfying $A N + N^\mathsf{T} A= 0$. Then
\begin{enumerate}
\item If $n$ is even, then $A$ has the form \eqref{A: n=n} with $a_{2l,n} = 0$ for all $l$.
\item If $n$ is odd, then $A$ has the form \eqref{A: n=n} with $a_{2l+1,n} = 0$ for all $l$.
\end{enumerate}
\end{Proposition}
\begin{Proposition}\label{P:anti-symmetric regular case}
If $A=(a_{i,j})_{i,j=1}^n$ is a skew-symmetric matrix satisfying $A N + N^\mathsf{T} A= 0$, then
\begin{enumerate}
\item If $n$ is even, then $A$ has the form \eqref{A: n=n} with $a_{2l+1,n} = 0$ for all $l$.
\item If $n$ is odd, then $A$ has the form \eqref{A: n=n} with $a_{2l,n} = 0$ for all $l$.
\end{enumerate}
\end{Proposition}

\subsection{One-parameter subgroups}
\label{SS:sl2triplet}

Let $T\subseteq \mt{SL}_n$ be a maximal torus.
Define a multiplicative one parameter subgroup (called {\em 1-PSG}, for short)
$\phi_n : \C^*  \rightarrow T \subset \mt{SL}_n$ by

\vspace{.4cm}

\noindent if $n=2k+1$, then $\phi_n (s) =:diag(s^k,\dots, s,1,s^{-1},\dots,s^{-k} )=diag(s^{k-i+1})_{i=1}^{2k+1}$;

\noindent if $n=2k$, then $\phi_n (s) =diag(s^{2k-1},s^{2k-3},\dots, s^1,s^{-1},\dots,s^{-(2k-1)} )= diag(s^{2k-(2i-1)})_{i=1}^{2k}$.

\vspace{.4cm}

Let $t\in \C$, and define $u_n(t): = \exp (tN_{(n)}) \in \mt{SL}_n$, an additive 1-PSG.
It is straightforward to calculate
\begin{align*}
u_n(t)=
\begin{pmatrix}
1 & t & \frac{t^2}{2!} &  \cdots & \frac{t^{n-2}}{(n-2)!} & \frac{t^{n-1}}{(n-1)!} \\
0 & 1 & t & \cdots & \frac{t^{n-3}}{(n-3)!} & \frac{t^{n-2}}{(n-2)!} \\
0 & 0 & 1 & \cdots & \frac{t^{n-4}}{(n-4)!} & \frac{t^{n-3}}{(n-3)!}\\
\vdots & \vdots & \vdots & \ddots & \vdots & \vdots  \\
0 & 0 & 0 & \dots & 1 & t \\
0 & 0 & 0 & \dots & 0 & 1
\end{pmatrix}.
\end{align*}
A straightforward matrix multiplication shows that $\phi_n(s) u_n(t) (\phi_n(s))^{-1}$ is equal to $u_n(st)$ if $n$ is odd and is equal to $u_n(s^2t)$ if $n$ is even.

\begin{Remark}
Let $N$ be an arbitrary nilpotent matrix with entries in $\C$, $u = \exp N$, and $U = \{ \exp(t N) : t \in \C \}$.
If $X$ is any complete variety on which $U$ acts, then $X^u = X^U$.
\end{Remark}

\subsection{Embeddings}
\label{SS:embeddings}

Let $G$ denote $\mt{SL}_n$, $B\subset G$ the subgroup of upper triangular matrices and let $T\subset B$ be the subgroup of diagonal matrices.
A closed subgroup $P$ of $G$ is called a {\em standard parabolic subgroup}, if $B\subseteq P$.
More generally, a subgroup $P'$ is called {\em parabolic}, if there exists $g\in G$ such that $P'=g Pg^{-1}$
for some standard parabolic subgroup $P\subset G$.

There exists a one-to-one correspondence between the standard parabolic subgroups of $G$ and the subsets
of $[n-1]$ given by:
$$
I\rightsquigarrow P_I:= B W_I B,
$$
where $W_I$ is the subgroup of the symmetric group $S_n$ generated by the simple transpositions $s_i = (i, i+1)$ for $i \in I$.
With this notation, the quotient space $G/P_I$ is called the {\em partial flag variety} of type $I$.
We need the following interpretation of $G/P_I$.

Let $I^c=\{j_1,\dots, j_{s}\}$ be the complement of $I$ in $[n-1]$.
Then $G/P_I$ is identified with the set of all nested sequences of vector spaces (flags) of the form
$$
\mc{F}:\ 0 \subset V_{j_1} \subset V_{j_2} \subset \cdots \subset V_{j_s} \subset \C^n,
$$
where $\dim V_{j_k} = j_k$, $k=1,\dots, s$.
The {\em type} of a flag $\mc{F}$ is defined to be the set $\{ j_1,\dots, j_s\}$, or equivalently, it is the composition $(j_1, j_2 - j_1, \dots, j_s - j_{s-1}, n - j_s)$ of $n$.

Finally, let us mention the fact that $G/P_I$ has a canonical decomposition into $B$-orbits where the orbits are indexed by
the elements of $W^I$, the minimal length coset representatives.
Furthermore, $T$ acts on $G/P_I$ and each $B$-orbit contains a unique $T$-fixed flag.
For more details, see \cite{Humphreys}.

We briefly review the theory of wonderful embeddings, referring the reader to \cite{DP83} for more details.

For us, a symmetric space is a quotient of the form $G/K$, where $G$ is an algebraic group and $K$ is the normalizer of
the fixed subgroup of an automorphism $\sigma: G \rightarrow G$ of order 2.
When $G$ is a semi-simple simply connected algebraic group (over an algebraically closed field), there exists a unique minimal
smooth projective $G$-variety $X=X_{G/K}$, called the {\em wonderful embedding} of $G/K$, such that
\begin{enumerate}
\item $X$ contains an open $G$-orbit $X_0$ isomorphic to $G/K$;
\item $X - X_0$ is the union of finitely many $G$-stable smooth codimension one subvarieties $X_i$ for $i = 1, 2, \dots, r$;
\item for any $I \subset [r]$, the intersection $X^I := \bigcap_{i \notin I} X_i$ is smooth and transverse;
\item every irreducible $G$-stable subvariety has the form $X^I$ for some $I \subset [r]$;
\item for $I\subset [r]$, let $\ms{O}^I$ denote the dense $G$-orbit in $X^I$.
There is a fundamental decomposition
\begin{equation}\label{E:orbit union}
X^I = \bigsqcup_{K \subset I} \ms{O}^K.
\end{equation}
Consequently, $G$-orbit closures form a Boolean lattice and there exists a unique closed orbit $Z$ corresponding to $I = \emptyset$.

\end{enumerate}

The recursive structure of wonderful embeddings is apparent from the following observation:

For each $I\subset [r]$, there exists a parabolic subgroup $P_I$ and $G$-equivariant fibrations $\pi: \ms{O}^I \rightarrow G/P_I$
and $\overline{\pi}_I : X^I \rightarrow G/P_I$ such that the following diagram commutes:

\begin{figure}[htp]
\centering
\begin{tikzpicture}[scale=.45]
\begin{scope}
\node at (-4,0) (a) {$\ms{O}^I$};
\node at (4,0) (b) {$X^I$};
\node at (0,-4.5) (c) {$G/P_I$};
\node at (-2.5,-2.5) (d) {$\pi_I$};
\node at (2.6,-2.5) (d) {$\overline{\pi}_I$};
\end{scope}
\begin{scope}
\draw[right hook->, thick] (a) to (b);
\draw[->, thick] (a) to (c);
\draw[->, thick] (b) to (c);
\end{scope}
\end{tikzpicture}
\end{figure}
\noindent
It turns out that the fiber of $\pi$ is $L/N$, where $L$ is the semi-simple part of the Levi subgroup of $P_I$ and $N$ is the
normalizer in $L$ of the fixed subgroup of $\theta$ induced on $L$.
Moreover, the fiber $(\overline{\pi}_I)^{-1}(x)$ over a point $x \in G/P_I$ is isomorphic to a wonderful embedding of $\pi_I^{-1}(x)$.

There is more to be said about the parametrizing set $[r]$ of the $G$-orbits.
We choose a maximal torus $S$ consisting of semi-simple elements $x\in G$ such that $\theta(x) = x^{-1}$.
Let $T$ be a maximal torus containing $S$. Then $T$ is automatically $\theta$-stable.
Denote by $\Phi$ the root system for $T$ in $G$. Then $\theta$ induces an involution on $\Phi$, which we denote by
$\theta$, also.
Let $\Phi_0 \subseteq \Phi$ denote the set of roots that are fixed by $\theta$, and let $\Phi_1$ denote the complementary
set of moved roots.
We choose a Borel subgroup $B$ containing $T$ such that the corresponding set of positive roots $\Phi^+$
satisfy $\theta (\Phi^+ \cap \Phi_1 ) \subseteq - \Phi^+$.

There is a natural restriction map $\text{res}: X^*(T) \rightarrow X^*(S)$ and the image of $\Phi_1$ gives a not necessarily reduced
root system in $X^*(S) \otimes_\Z \R$.
Let $\Delta$ denote the set of simple roots of $\Phi^+$ and set $\Delta_1 = \Delta \cap \Phi_1$.
We call the image $\overline{\Delta} = \text{res}(\Delta_1)$, the set of {\em simple restricted roots.}

Finally, let $\Pi$ denote the set of fixed simple roots $\Phi_0 \cap \Delta$. We denote by $P_\Pi$ the corresponding
parabolic subgroup of $G$.
The unique closed orbit is isomorphic to $G/P_\Pi$.


\subsection{Example: Complete Quadrics}\label{S:complete quadrics}

There is a vast literature on the variety of complete quadrics.  See \cite{Laksov87} for a survey.
We briefly recall the necessary definitions.

Two non-degenerate $n\times n$ symmetric matrices represent the same non-degenerate quadric hypersurface
in $\PP^{n-1}$ if and only if they differ by multiplication by a scalar; hence $\PP(\Sym_n)$ is the moduli of such
quadric hypersurfaces.
In this context, the classical definition of the variety of complete quadrics in $\PP^{n-1}$
(see \cite{Schubert, Semple48, Tyrrell56}) is as the closure of the image of the map
\begin{align}\label{SST}
[A] \mapsto ([\Lambda^1(A)], [\Lambda^2(A)], \dots, [\Lambda^{n-1}(A)]) \in \prod_{i=1}^{n-1} \PP(\Lambda^i(\text{Sym}_n)), \ [A]\in \PP(\Sym_n^0).
\end{align}
In more modern group-theoretical language, the variety of complete quadrics is the wonderful embedding 
$\mathcal{M}_{\text{Q}}=\mathcal{M}_{\mt{SL}_n / Z(\mt{SL}_n) \mt{SO}_n}$ of the symmetric space $\mt{SL}_n / Z(\mt{SL}_n) \mt{SO}_n$.

Let $I\subseteq [n-1]$ and let $P_I$ be the corresponding standard parabolic subgroup.
It follows from results of Vainsencher \cite{Vainsencher82} that a point $\mathcal{P} \in \mathcal{M}_{\text{Q}}$ is described by the data of a flag
\begin{equation}\label{E:flag}
\mathcal{F}: V_0 = 0 \subset V_1 \subset \dots \subset V_{s-1} \subset V_s =\C^n
\end{equation}
and a collection $\mathcal{Q} = (Q_1, \dots Q_s)$ of quadrics $Q_i \subseteq \PP(V_i)$ whose singular locus
is $\PP(V_{i-1})$.  Moreover, $\mc{O}^I$ consists of complete quadrics whose flag $\mathcal{F}$ is of type $I$
and the map $(\mathcal{F}, \mathcal{Q}) \mapsto \mathcal{F}$ is the $\mt{SL}_n$-equivariant projection
\begin{equation}\label{E:pi_K}
\overline{\pi}_I : \mathcal{M}_{\text{Q}}^I \rightarrow \mt{SL}_n / P_I.
\end{equation}
The fiber of $\overline{\pi}_I$ over $\mathcal{F} \in \mt{SL}_n / P_I$ is isomorphic to a product of varieties of complete quadrics of
smaller dimension.

Let $\theta_o : \mt{SL}_n \rightarrow \mt{SL}_n$ denote the automorphism $\theta_o (g) = (g^{-1})^\top$. Then $\theta_o$ is
an involution and its fixed subgroup is precisely $H=\mt{SO}_n$.
In this case, the set of moved roots and the set of simple roots $\Delta$ of $T$ coincide. Here $T$ is the
maximal torus of diagonal matrices in $\mt{SL}_n$.
The unique closed orbit, therefore, corresponds to the empty set of roots, and hence, it is isomorphic to
the flag variety $G/B$.

Furthermore, there is a one-to-one correspondence between the subsets of $\Delta_1 = \Delta$ and the $G$-orbits.
It follows for each standard parabolic subgroup $P_I\subseteq G$, $I\subset \Delta$, there exists
$G$-equivariant fibration $\pi: \ms{O}^I \rightarrow G/P_I$ with fibers isomorphic to $L_I/N_I$,
where $L_I$ is the semi-simple part of the Levi subgroup of $P_I$ and $N_I$ is the
normalizer in $L_I$ of the fixed subgroup of $\theta$ induced on $L_I$.

On the other hand, we know that the standard Levi subgroups in $G=\mt{SL}_n$ are of the form $\prod_{j=1}^k \mt{SL}_{m_j}$,
where $\sum m_j = n$. Since the $\theta$-fixed subgroup of such a group is isomorphic to $\prod_{j=1}^k \mt{SO}_{m_j}$,
a fiber of $\pi$ has to form $\prod_{j=1}^k \Sym_{m_j}^0$.

\subsection{Example: Complete Skew Forms}\label{S:complete skew forms}

{\em Notation:} From now on, $n$ is assumed to be an even number whenever we deal with a symplectic group.

The classical construction of the wonderful compactification of $\mt{SL}_{n} / Z(\mt{SL}_{n}) \mt{Sp}_{n}$, 
$\mathcal{M}_{\text{S}}=\mathcal{M}_{\mt{SL}_{n} / Z(\mt{SL}_{n}) \mt{Sp}_{n}}$, is similar to that of $\mathcal{M}_{\text{Q}}$; 
instead of symmetric matrices in (\ref{SST}) one needs to use 2-forms ($n \times n$ skew-symmetric matrices).

The group-theoretical definition parallels that of $\mathcal{M}_{\text{Q}}$ as well.
A point $\mathcal{P} \in \mathcal{M}_{\text{S}}$ is given by the data of a flag of even dimensional subspaces
\begin{equation}\label{E:flag}
\mathcal{F}: V_0 = 0 \subset V_1 \subset \dots \subset V_{s-1} \subset V_s =\C^n
\end{equation}
and a collection $\mathcal{W} = (\omega_1, \dots \omega_s)$ of 2-forms, where $\omega_i$ is a 2-form in $\PP(V_i)$ whose
null space is $\PP(V_{i-1})$. Moreover, $\ms{O}^I$ consists of complete 2-forms whose flag $\mathcal{F}$ is of type $I$
and the map $(\mathcal{F}, \mathcal{W}) \mapsto \mathcal{F}$ is the $\mt{SL}_{n}$-equivariant projection
\begin{equation}\label{E:pi_K}
\overline{\pi}_I : \mathcal{M}_{\text{S}}^I \rightarrow \mt{SL}_{n} / P_I.
\end{equation}
The fiber of $\overline{\pi}_I$ over $\mathcal{F} \in \mt{SL}_{n} / P_I$ is isomorphic to a product of wonderful 
compactifications of spaces of complete skew forms of smaller dimension.

\section{\textbf{Fixed Subvarieties}}
\label{S:fixedsubvarieties}

In this section we describe the fixed subvarieties of a regular unipotent element operating on
complete quadrics and complete skew forms.

\begin{Definition}\label{dfn:u fixed and N stable flags}
Let
$$
\mathcal{F} : 0 = V_0 \subset V_1 \subset V_2 \subset \cdots \subset V_k = V
$$
be a (partial) flag of $V = \mathbb{C}^n$.  For any $g \in \mt{SL}_n$, the flag $\mathcal{F}$ is $g$-fixed if and only if 
$g(V_i) = V_i$ for all $1 \leq i \leq k$.  For any $n \times n$ matrix $M$, the flag $\mathcal{F}$ is $M$-stable if and only if 
$M(V_i) \subseteq V_{i-1}$ for all $1 \leq i \leq k$.
\end{Definition}

\begin{Remark}
Let $N$ be a nilpotent $n \times n$ matrix and $u = \exp(N)$ the corresponding unipotent element of $\mt{SL}_n$. 
Then a flag $\mathcal{F}$ of $\mathbb{C}^n$ is $u$-fixed if and only if it is $N$-stable.
\end{Remark}

\begin{Lemma}\label{lemma:u fixed flags}
Let $V = \C^n$ with standard basis $\{ e_1, e_2, \dots, e_n \}$, and let
$$
N = N_{(n)} = \begin{pmatrix} 0 & 1 & 0 & \cdots & 0 \\ 0 & 0 & 1 & \cdots & 0 \\ \vdots & \vdots & \vdots & \ddots & \vdots \\ 0 & 0 & 0 & \cdots & 0 \end{pmatrix}
$$
be the regular nilpotent operator for this basis. Set $u = \exp(N)$.
Then for each composition $\gamma$ of $n$, there is a unique flag of type $\gamma$ that is fixed by $u$,
namely the sub-flag of type $\gamma$ of the standard complete flag
$0 \subset \C e_1 \subset \C e_1 \oplus \C e_2 \subset \cdots \subset \C e_1 \oplus \C e_2 \oplus \dots \oplus \C e_n = V$.
\end{Lemma}

\begin{proof}
We prove this by induction on the number $k$ of parts of the composition $\gamma$.
Let $v$ be any non-zero vector in $V_1$ and let $j$ be the largest integer such that $v$ contains a non-zero $e_j$ component.
Then $v, N(v), N^2(v), \dots, N^{j-1}(v)$ are $j$ linearly independent vectors in $V_1$.
Since $\dim(V_1) = \gamma_1$, this implies $j \leq \gamma_1$.
Thus $V_1 = \C e_1 \oplus \C e_2 \oplus \dots \oplus \C e_{\gamma_1}$.
Now apply the inductive hypothesis to the induced flag obtained by quotienting
$\mathcal{F}$ by $V_1$, noting that the induced nilpotent transformation on $V / V_1$ has
Jordan type $(n - \gamma_1)$, allowing the induction to proceed.
\end{proof}

\begin{Remark}\label{R:evenorodd}

Let $u= \exp(N)$ be as in Lemma \ref{lemma:u fixed flags}, and let
$\mathcal{F}: 0 \subset V_1 \subset V_2 \subset \dots \subset V_k = \C^n$ be an $N$-stable flag (or, equivalently, 
$u$-fixed) of type $\gamma =(\gamma_1,\dots, \gamma_k)$.
Then restriction of $N$ to $V_i / V_{i - 1}$ is another principal nilpotent operator and its Jordan type is $(\gamma_i)$.
It follows from Section \ref{SS:unipotentinvariantmatrices} that, if $\dim (V_i / V_{i-1})$ is even for some
$1\leq i \leq k$, then there are no $u$-fixed quadrics in $\PP(V_i / V_{i-1})$. Consequently,
there cannot be any $u$-fixed complete quadrics on such a flag $\mc{F}$.
On the other hand, if $\dim (V_i / V_{i-1})$ is odd for some
$1\leq i \leq k$, then there are no $u$-fixed 2-forms in $\PP(V_i / V_{i-1})$, hence no $u$-fixed complete skew-forms on such a flag.

\end{Remark}

We are ready to give a set theoretic description of the $u$-fixed subvarieties of our compactifications.

\begin{Proposition}\label{P:decompositionintoaffinesets}
The unipotent fixed subvariety $\mathcal{M}_{\text{Q}}^{(n)}$ has a decomposition into affine spaces indexed
by the compositions of $n$ with no even parts. Similarly, the unipotent fixed subvariety $\mathcal{M}_{\text{S}}^{(n)}$
has a decomposition into affine spaces indexed by the compositions of $n$ with no odd parts.
\end{Proposition}

\begin{proof}

We begin with the case $\mathcal{M}_{\text{Q}}^{(n)}$.
Let $(\mc{F},\mc{Q})$ be a point of $\mathcal{M}_{\text{Q}}^{(n)}$ with $\mc{F}$ of type $\gamma =(\gamma_1,\dots, \gamma_k)$.
Then by Remark \ref{R:evenorodd} we know that each $\gamma_i$ ($1\leq i \leq k$) is odd.
It follows from Proposition \ref{P:symmetric regular case} (1) that the space of non-degenerate $u$-fixed quadrics on 
$\PP(V_i / V_{i - 1})$ is an affine space of dimension $(\gamma_i - 1)/2$.

Let $\mathcal{M}_{\text{Q}}^{(n),\gamma} \subseteq \mathcal{M}_{\text{Q}}^{(n)}$ denote the subspace
consisting of complete quadrics of flag type $\gamma$. Clearly,
$$
\mathcal{M}_{\text{Q}}^{(n)} = \bigsqcup_{\gamma} \mathcal{M}_{\text{Q}}^{(n),\gamma},
$$
where the union is over all compositions $\gamma = (\gamma_1,\dots, \gamma_k)$ of $n$ with odd parts only.
Moreover, the above discussion implies that $\mathcal{M}_{\text{Q}}^{(n),\gamma}$ is an affine space:
$$
\mathcal{M}_{\text{Q}}^{(n),\gamma} = \A^{(\gamma_1-1)/2} \times \cdots \times \A^{(\gamma_k-1)/2} \cong \A^{(n-k)/2}.
$$
The argument for $\mathcal{M}_{\text{S}}^{(n)}$ is similar, so we write the end result:
$$
\mathcal{M}_{\text{S}}^{(n)} = \bigsqcup_{\beta} \mathcal{M}_{\text{S}}^{(n),\beta},
$$
where the union is over all compositions $\beta = (\beta_1,\dots, \beta_k)$ of $n$ with even parts only,
and each $\mathcal{M}_{\text{S}}^{(n),\beta}$ is an affine space:
$$
\mathcal{M}_{\text{S}}^{(n),\beta} = \A^{\beta_1/2-1} \times \cdots \times \A^{\beta_k/2-1} \cong \A^{n/2-k}.
$$
\end{proof}

As an immediate corollary of the proof of Proposition \ref{P:decompositionintoaffinesets} we have
\begin{Corollary}
$\dim \mathcal{M}_{\text{Q}}^{(n)}  = \lfloor (n-1)/2 \rfloor$, and $\dim \mathcal{M}_{\text{S}}^{(n)} = n/2 -1$.
\end{Corollary}

\section{\textbf{Cell Decompositions}}
\label{S:celldecompositions}

\subsection{$\C^*$ action on matrices}
\label{SS:torusactiononmatrices}

Let $T\subset \mt{SL}_n$ denote the maximal torus of diagonal matrices.
We compute the stabilizer of $T$ acting on $X^{(n)}$, where $X$ is any of the spaces $X=\Mat_n,X= \Skew_n$, or $X=\Sym_n$.

Let $Y$ denote the set of invertible elements from $X$.
By \eqref{A: n=n} and Propositions \ref{P:symmetric regular case} and \ref{P:anti-symmetric regular case}, the set $Y^{(n)}$ is non-empty 
if and only if the main anti-diagonal of a generic element from $Y$ has no zero entries.
If $Y= \mt{GL}_n$, this is independent of $n$; however, by Propositions \ref{P:symmetric regular case} and \ref{P:anti-symmetric regular case}, 
if $Y= \Sym_n^0$, then $n$ has to be odd, and if $Y=\Skew_n^0$, then $n$ must be even.

Let $t \in T$ be such that $t^{-1} = \text{diag} ( t_1,\dots, t_n )$, $t_i \in \C^*$, and
let $A= (a_{i,j}) \in Y^{(n)}$ be a generic element. It is straightforward to verify that
\begin{align}\label{pairwise}
t\cdot A = (t_i t_j a_{i,j})_{i,j=1}^n \in Y^{(n)} \iff t_i t_j = t_k t_l\ \text{for all} \ i+j = k+l.
\end{align}

\begin{Example}
Let $A$ be a generic element of $X^{(6)}=\Mat_6^{(6)}$. Then
\begin{equation*}\label{t dot A}
t \cdot A = (t^{-1})^\top A (t^{-1})=
\begin{pmatrix}
0 & 0 & 0 & 0 & 0 & t_1 t_6 a_{11} \\
0 & 0 & 0 & 0 & -t_2 t_5 a_{11} & t_2 t_6 a_{12} \\
0 & 0 & 0 & t_3 t_4 a_{11} & -t_3 t_5 a_{12} & t_3 t_6 a_{13} \\
0 & 0 & -t_4 t_3 a_{11} & t_4^2 a_{12} & -t_4 t_5 a_{13} & t_4 t_6 a_{14} \\
0 & t_5 t_2 a_{11} & -t_5 t_3 a_{12} & t_5 t_4 a_{13} & -t_5^2 a_{14} & t_5 t_6 a_{15} \\
-t_6 t_1 a_{11} & t_6 t_2 a_{12} & -t_6 t_ 3 a_{13} & t_6 t_4 a_{14} & -t_6 t_5 a_{15} & t_6^2 a_{16}
\end{pmatrix}.
\end{equation*}
\end{Example}

\vspace{1cm}

Let $\phi_n : \C^*  \rightarrow T \subset \mt{SL}_n$ denote the 1-PSG defined in Section \ref{SS:sl2triplet}.
Observe that if $a_{i,j}$ is the $(i,j)$-th entry of $A\in X$, then the $(i,j)$-th entry of
$\phi_n(s)\cdot A$ is $s^{i+j - (n+1)} a_{i,j}$.
Therefore, both $X^{(n)}$ and $Y^{(n)}$ are stable under $\C^*$ action defined by $\phi_n(s)$.
Moreover, $A\in Y^{(n)}$ is a fixed point of this $\C^*$-action if and only if the only non-zero entries of $A$ appear along its main anti-diagonal.

\begin{Example}
Let $\phi_6(s) = \text{diag} (s^{5},s^{3},s^{1},s^{-1},s^{-3},s^{-5})$, and
let $A\in X^{(6)}$, $X=\Skew_6$ denote a generic element, which is of the form
\begin{equation*}
A=
\begin{pmatrix}
0 & 0 & 0 & 0 & 0 & a_{11} \\
0 & 0 & 0 & 0 & -a_{11} & 0 \\
0 & 0 & 0 & a_{11} & 0 & a_{13} \\
0 & 0 & -a_{11} & 0 & -a_{13} & 0  \\
0 & a_{11} & 0 &  a_{13} & 0 &  a_{15} \\
-a_{11} & 0 &  -a_{13} &  0 &  -a_{15} & 0
\end{pmatrix}.
\end{equation*}
Then
\begin{align*}
s \cdot A = \phi(s^{-1})^\top A \phi(s^{-1})=\begin{pmatrix}
0 & 0 & 0 & 0 & 0 & a_{11} \\
0 & 0 & 0 & 0 & -a_{11} & 0 \\
0 & 0 & 0 & a_{11} & 0 & s^{4}a_{13} \\
0 & 0 & -a_{11} & 0 & -s^{4}a_{13} & 0  \\
0 & a_{11} & 0 &  s^{4}a_{13} & 0 &  s^8 a_{15} \\
-a_{11} & 0 &  -s^{4}a_{13} &  0 &  -s^8 a_{15} & 0
\end{pmatrix}.
\end{align*}

\end{Example}

\vspace{1cm}

\begin{Lemma}\label{L:torusfixed-classicalcase}
Let $\phi_n : \C^* \rightarrow T$ be defined as in Section \ref{SS:sl2triplet}.
Let $Y$ be either $\Skew_n^0$ or $\Sym_n^0$.
Then $\C^*$ acts on $\PP(Y^{(n)})$ via $\phi_n$ with a unique fixed point.
Moreover, the unique $\C^*$-fixed point $x \in \PP(Y^{(n)})^{\C^*}$ is
$$
x = \begin{bmatrix}
0 & 0 &  \cdots & 0 & 1 \\
0 & 0 & \cdots  & -1& 0 \\
\vdots &  \vdots &  \iddots & \vdots & \vdots \\
0 & 1 &  \cdots &  0 & 0 \\
-1 & 0&  \cdots & 0 & 0
\end{bmatrix}.
$$
\end{Lemma}

\subsection{Carrell-Goresky's Bia{\l}ynicki-Birula cell decomposition}

Under a ``good'' $\C^*$-action, a variety decomposes into affine subspaces, and the classes of the closures of these affine subspaces form a basis
for the Chow ring. We briefly recall certain aspects of this well developed theory of torus actions.

\begin{Theorem}\cite[Theorem 4.3]{BB73}\label{thm:BBdecomp}
Let $Z$ be a smooth, irreducible projective variety with a $\C^*$-action such that the fixed point set is finite.
Then
\begin{enumerate}
\item For each fixed point $x \in Z^{\C^*}$, the set
$$
C_x := \{ y \in Z : \lim_{t \rightarrow 0} t \cdot y = x \}
$$
is an affine space.\label{bb item 1}
\item The closures $\overline{C_x}$ of the cells $C_x$ form an additive basis for the
Chow ring (as well as the integral cohomology ring) of $Z$.\label{bb item 2}
\end{enumerate}
\end{Theorem}

It turns out that the smoothness condition in the above theorem can be relaxed. 	
\begin{Theorem}\cite[Theorem 1]{CG83}\label{thm:CG}
Let $Z$ be a possibly singular, possibly reducible $\C^*$-variety whose fixed point set is finite.
If the first statement in Theorem \ref{thm:BBdecomp} holds, then so does the second statement.
\end{Theorem}

If the first statement of Theorem \ref{thm:BBdecomp} holds (hence the conclusion of Theorem \ref{thm:CG} holds), then
we call the decomposition $Z = \bigsqcup_{x \in Z^{\C^*}} C_x$ {\em a cell decomposition} of $Z$ with respect to the $\C^*$-action,
and refer to the affine sets $C_x$, $x\in Z^{\C^*}$, as the {\it cells} of the decomposition.

\subsection{Torus action on compactifications}

We consider extensions of the action given by $\phi_n$ (defined in Section \ref{SS:sl2triplet}) on matrices to $\mathcal{M}_{\text{Q}}$ and to $\mathcal{M}_{\text{S}}$.
Recall from the proof of Proposition \ref{P:decompositionintoaffinesets} that there are decompositions
\begin{align}\label{A:decompositions}
\mathcal{M}_{\text{Q}}^{(n)} = \bigsqcup_{\gamma} \mathcal{M}_{\text{Q}}^{(n),\gamma} \ \text{ and } 
\ \mathcal{M}_{\text{S}}^{(n)} = \bigsqcup_{\beta} \mathcal{M}_{\text{S}}^{(n),\beta},
\end{align}
where $\mathcal{M}_{\text{Q}}^{(n),\gamma}$ is the affine space consisting of unipotent invariant complete quadrics of flag type $\gamma$,
and $\mathcal{M}_{\text{S}}^{(n),\beta}$ is the affine space consisting of unipotent invariant complete skew-forms of flag type $\beta$.

\begin{Theorem}
Let $\phi_n$ be the 1-PSG defined in Section \ref{SS:sl2triplet}.  Then $\C^*$ acts on the spaces
$\mathcal{M}_{\text{Q}}^{(n)}$ and $\mathcal{M}_{\text{S}}^{(n)}$ with finitely many fixed points.
The cells of the torus action are given by the decompositions (\ref{A:decompositions}).
\end{Theorem}

\begin{proof}

Let $e_1,\dots,e_n$ denote the standard ordered basis for $\C^n$, and let
$\mc{F}_s$ denote the standard flag $0\subset V_1 \subset V_2 \subset \cdots \subset V_n= \C^n$,
where $V_r$ is the $\C$-span of $e_1,\dots, e_r$.
Let $x=(\mc{F},\mc{P})$ be either from $\mathcal{M}_{\text{Q}}^{(n)}$ or from $\mathcal{M}_{\text{S}}^{(n)}$.
By Lemma \ref{lemma:u fixed flags}, we know that in order for $x$ to be fixed by $u= \exp N_{(n)}$
its flag $\mc{F}$ has to be a subflag of  $\mc{F}_s$.
Suppose $\mc{F}: \ 0 \subset V_{j_1} \subset \cdots \subset V_{j_k} = \C^n$, and suppose
$\mc{P}=(P_1,\dots, P_k)$ is the (symmetric or skew) form of $x$.
Then the $i$-th entry of $\mc{P}$ is a non-degenerate form (up to a scaler multiple) on the quotient space
$V_{j_i}/V_{j_i-1}$, and furthermore, it is fixed by the restriction of $u$ on $V_{j_i}/V_{j_i-1}$, which is a
regular unipotent operator of Jordan type $(\dim V_{j_i}-\dim V_{j_{i-1}})$.

It is straightforward to verify that $\phi_n (\C^*) \cdot \mc{F}= \mc{F}$.
Therefore, it remains showing that there are only finitely many $\mc{P}$ on $\mc{F}$ such that
$\phi(\C^*) \cdot \mc{P} = \mc{P}$. Note that this equality is true if and only if
the restriction of $\phi_n(\C^*)$ to each quotient space $V_{j_i}/V_{j_{i-1}}$ acts on the corresponding form $P_i$ trivially.
But we know from Lemma \ref{L:torusfixed-classicalcase} that
$P_i$ has to be represented by an anti-diagonal matrix of size $\dim V_{j_i}-\dim V_{j_{i-1}}$.
In either case, our conclusion is that there is only one $\C^*$-fixed point for each subflag $\mc{F}$ of
the standard flag $\mc{F}_s$.
Since $\C^*$-action preserves the flag type, and since the affine subsets
$\mathcal{M}_{\text{Q}}^{(n),\gamma}$, $\mathcal{M}_{\text{S}}^{(n),\beta}$ from (\ref{A:decompositions}) are
precisely the unipotent invariant complete forms of a given flag type,
by Theorem \ref{thm:CG} we see that the affine subsets of (\ref{A:decompositions}) must be the cells of the $\C^*$-action.

\end{proof}

\vspace{.5cm}

Let $Comp_{o}(n)$ denote the set of all compositions of $n$ with odd parts only,
and let $Comp_{e}(n)$ denote the set of all compositions of $n$ with even parts only.
\begin{Corollary}
The Poincar\'e polynomial $P_{\mathcal{M}_{\text{Q}}^{(n)}} (x)$ of $\mathcal{M}_{\text{Q}}^{(n)}$ is given by
$$
P_{\mathcal{M}_{\text{Q}}^{(n)}} (x) = \sum_{(\gamma_1,\dots,\gamma_k) \in Comp_o(n)} x^{ (n-k)/2 } =  \sum_{i=0}^n {n - 1 - i \choose i} x^i.
$$
\end{Corollary}
\begin{proof}
For simplicity let $p_n(x)$ denote the Poincar\'e polynomial $P_{\mathcal{M}_{\text{Q}}^{(n)}} (x)$.
The first equality follows from the proof of Proposition \ref{P:decompositionintoaffinesets}.
To prove the second equality, we first find a recurrence for $p_n(x)$. To this end,
observe that $Comp_o(n) = Comp_o(n)' \sqcup Comp_o(n)''$,
where $Comp_o(n)'$ consists of elements
$\gamma \in Comp_o(n)$ such that $\gamma_1 = 1$, and $Comp_o(n)''$ consists of those $\rho \in Comp_o(n)$
with $\rho_1 > 1$.
\begin{align*}
p_n(x) &=
\sum_{\gamma = (\gamma_1,\dots,\gamma_k) \in Comp_o(n)''} x^{ (n - k)/2  } +
\sum_{\rho=( \rho_1 \dots, \rho_l) \in Comp_o(n)''} x^{ (n - l)/2  }\\
&= \sum_{(\gamma_2,\dots,\gamma_k) \in Comp_o(n-1)} x^{ [(n-1) - (k-1)]/2  } +
\sum_{( \rho_1 -2 \dots, \rho_l) \in Comp_o(n-2)} x \cdot x^{ [(n - 2) - l]/2  },
\end{align*}
which simplifies to
\begin{equation}\label{A:recurrenceofp_n}
p_n(x) = p_{n-1}(x) + x p_{n-2} (x).
\end{equation}
It is easy to easy to check that $p_1(x) = p_2(x) = 1$.
Now, a simple induction argument combined with the well known recurrence
$$
{n - 1 - i \choose i} = {n - 2 - i \choose i - 1} + {n - 2 - i \choose i}
$$
together with (\ref{A:recurrenceofp_n}) give us the desired second equality.

\end{proof}

\begin{Remark}
It follows from (\ref{A:recurrenceofp_n}) and the initial conditions that the value at $x=1$ of the Poincar\'e
polynomial $P_{\mathcal{M}_{\text{Q}}^{(n)}}(x)$ is the $(n-1)$-st Fibonacci number.
\end{Remark}

\begin{Corollary}
Suppose $n=2m$.
Then the Poincar\'e polynomial $P_{\mathcal{M}_{\text{S}}^{(n)}} (x)$ of $\mathcal{M}_{\text{S}}^{(n)}$ is given by
$$
P_{\mathcal{M}_{\text{S}}^{(n)}} (x) = \sum_{(\beta_1,\dots,\beta_k) \in Comp_e(n)} x^{ n/2-k } = (1+x)^{m-1}.
$$
\end{Corollary}
\begin{proof}
The first equality follows from the proof of Proposition \ref{P:decompositionintoaffinesets}.
To prove the second equality observe that $Comp_e(n)$ is in bijection with $Comp(m)$, the set of all
compositions of $m$. Thus, the equality
\begin{align}\label{A:transition}
P_{\mathcal{M}_{\text{S}}^{(n)}} (x) = \sum_{(\beta_1,\dots,\beta_k) \in Comp (m) } x^{ m-k }
\end{align}
is obvious.

Recall the bijection between the set of all compositions of $m$ and the set of all subsets
of $[m-1]$: For $\gamma = (\gamma_1,\dots, \gamma_k)$ a composition of $m$, let $I_\gamma$
denote the complement of the set $I=\{ \gamma_1,\gamma_1+\gamma_2,\dots, \gamma_1+\cdots+\gamma_{k-1}\}$
in $[m-1]$. Then the assignment $\gamma \mapsto I_\gamma$ is the desired bijection between compositions
of $m$ with $k$ parts and the subsets of $[m-1]$ with $m-1- (k-1) = m-k$ elements.
Therefore, the right hand side of equation (\ref{A:transition}) is equal to $\sum_{I\subseteq [m-1]} x^{|I|}$, which is
obviously equal to $(1+x)^{m-1}$.

\end{proof}

\bibliography{References}
\bibliographystyle{plain}

\end{document}